\documentclass[12pt]{amsart}

\RequirePackage[colorlinks,citecolor=blue,urlcolor=blue]{hyperref}

\usepackage{amsmath,amstext,amssymb,amsopn,amsthm}
\usepackage{url,verbatim}
\usepackage{mathtools}
\usepackage{enumerate}

\usepackage{color,graphicx}
\usepackage{subfig}

\usepackage[margin=30mm]{geometry}
\usepackage{eucal,mathrsfs,dsfont}

\allowdisplaybreaks

\newtheorem{theorem}{Theorem}[section]
\newtheorem{corollary}[theorem]{Corollary}
\newtheorem{lemma}[theorem]{Lemma}

\theoremstyle{definition}

\newtheorem{definition}[theorem]{Definition}

\newtheorem{remark}[theorem]{Remark}
\newtheorem{example}[theorem]{Example}

\numberwithin{equation}{section}

\newcommand{\eps}{\varepsilon}

\newcommand{\calF}{\mathcal{F}}

\newcommand{\calT}{\mathcal{T}}

\newcommand{\calB}{\mathcal{B}}
\newcommand{\calV}{\mathcal{V}}
\newcommand{\calE}{\mathcal{E}}
\newcommand{\calX}{\mathcal{X}}

\newcommand{\R}{\mathds{R}}
\newcommand{\C}{\mathds{C}}

\newcommand{\prt}{\partial}

\newcommand{\bF}{\mathbf{F}}
\newcommand{\be}{\mathbf{e}}

\newcommand{\ol}{\overline}

\newcommand{\wt}{\widetilde}

\DeclareMathOperator{\dist}{dist}

\def\n{{\bf n}}

\def \be {{\bf e}}

\title[Pinned billiard balls]{Upper bound on the number of collisions
\\ of pinned billiard balls}
\author{Krzysztof Burdzy and Mauricio Duarte}

\address{KB: Department of Mathematics, Box 354350, University of Washington, Seattle, WA 98195}
\email{burdzy@uw.edu}

\address{MD: Departamento de Matematicas, Facultad de Ciencias Exactas, Universidad Andres Bello, Santiago, Chile}
\email{mauricio.duarte@unab.cl}

\thanks{KB's research was supported in part by Simons Foundation Grant 506732. }
\thanks{MD's research was supported in part by Proyecto FONDECYT 1201639.}

\begin{document}

\begin{abstract}

We  consider systems of ``pinned balls,'' i.e.,  balls that  have fixed positions and pseudo-velocities. Pseudo-velocities change according to the same rules as those  for velocities of totally elastic collisions between moving balls. The times of possible pseudo-collisions for different pairs of pinned balls are chosen in an exogenous way.
We  give an explicit upper bound for the maximum number of pseudo-collisions for
a system of $n$ pinned balls in a $d$-dimensional space.  The proof is based on analysis of foldings, i.e., mappings that formalize the idea of folding a piece of paper along a crease. We prove an upper bound for the size of an orbit of a point subjected to foldings.

\end{abstract}

\maketitle

\section{Introduction}
\label{se:intro}

This paper is 
a follow-up to \cite{fold}. The main new results are an upper bound for the size of an orbit in a sequence of foldings and an upper bound for the number of collisions of pinned balls. In the latter case, we give two estimates; one of them depends on the number $n$ of balls and dimension of the space $d$; the other one  depends only on $n$ (see Theorem \ref{au21.7}). 
This is an improvement over the main result proved in \cite{fold}, where the bound depended also on the configuration of the pinned balls (we state that result in the review section as Theorem \ref{au21.6}). Our new bounds are comparable to the best known bounds for the number of collisions of  billiard balls moving and colliding according to the classical model---see Remark \ref{s11.2} (iv).

\subsection{Pinned Balls} 
The article \cite{fold} was inspired by articles on the maximum number of totally elastic collisions for a finite system of balls in a billiard table with no walls (i.e., the whole Euclidean space). We will review the history of this problem in Section \ref{review}. 

The main concern of \cite{fold} was the concept of ``pinned balls.'' In this model, balls have positions and pseudo-velocities, and some balls are in contact with other balls. The balls do not move, i.e., their positions are constant as functions of time. But the pseudo-velocities change according to the same rules as those for velocities of totally elastic collisions between moving balls, except that the order in which these collisions may occur is prescribed in an exogenous way.

Consider two balls capable of moving and assume that they touch at the initial time. Depending on their initial velocities, they may start moving according to these velocities ``away from each other,'' or they will experience an instant collision and will start moving with new velocities. Similarly, a pair of pinned (not moving) touching balls may retain its pseudo-velocities or their velocities will change, according to the same rules as for the moving balls. For this reason, if a pair of pinned balls is chosen and ``allowed'' to collide, their velocities may change or not. If the velocities change, we will say that a pseudo-collision occurred, although most of the time we will abbreviate ``pseudo-collisions'' to ``collisions.'' 

Given a (finite or infinite) sequence of pairs of touching pinned balls, the number of collisions along this sequence may be smaller than the length of the sequence. Our main result, Theorem \ref{au21.7}, gives an upper bound for the number of collisions that depends only on the number of balls and the dimension of the space.
This improves on the main result in \cite{fold} where a bound was proved but depended also  on the positions of the pinned balls.
Neither the bound in \cite{fold} nor the bound in Theorem \ref{au21.7} in the present paper
 depend on the initial pseudo-velocities or the order in which the balls are allowed to collide. 

The pinned balls model is inspired by the usual system of moving and colliding balls. Specifically, the main theorem in \cite{BD} (see Section \ref{review} below) is an example based in an essential way on the analysis of a pinned ball configuration.

As an additional motivation for studying pinned ball configurations, we point out that
there are reasons to think that  a large number of collisions for a system of moving balls can occur only when many of the balls form a tight configuration. Theorem 1.3 in \cite{BuD18b} states that if a family of $n$ balls undergoes more than $n^{cn}$ collisions for an appropriate $c>0$, then there is a subfamily $\calB'$ of balls and an interval of time $[t_1,t_2]$ in which a large number of collisions occur between members of $\calB'$, and the balls in $\calB'$ form a very tight configuration during the whole interval (see  \cite[Thm.~1.3]{BuD18b} for the quantitative version of the qualitative statement presented here).

The examples presented in
\cite{BuI18}, showing that the number of collisions of moving balls can be exponentially large, are based on tight ball configurations.

\subsection{Foldings}
\label{sec:fold}
  A \emph{folding} with respect to a halfspace $H$ in $\R^d$ is the identity on $H$ and maps the complementary halfspace $H'$ onto $H$ via reflection in the hyperplane $P$ which is the common boundary $\partial H' = \partial H$. It has been shown in \cite{fold} that collisions of pinned balls can be represented as foldings (see Section \ref{fold} below).
Consider an infinite sequence of foldings corresponding to a finite number of halfspaces whose intersection has a non-empty interior.  It has been proved in \cite{fold} that for every starting   point, its orbit generated by the sequence of foldings  is finite.
Our main technical result in the present paper, Theorem \ref{f11.3},  is an explicit bound on the orbit size in a sequence of foldings. The bound depends on the number of halfspaces and the size of a ball inside the intersection of the halfspaces, with the center at a distance less than 1 from the starting point. The bound does not depend on the dimension of the space.

\subsection{Hard ball collisions---historical review}
\label{review}

The question of whether a finite system of hard balls can have an infinite number of elastic collisions was posed by Ya.~Sinai. It was answered in negative in \cite{Vaser79}. For alternative proofs see \cite{Illner89, Illner90,IllnerChen,BuD18b}. 
It was proved in \cite{BFK1} that a system of $n$ balls in the Euclidean space undergoing elastic collisions can experience at most
\begin{align}\label{s26.1}
\left( 32 \sqrt{\frac{m_{\text{max}}}{m_{\text{min}}} } 
\frac{r_{\text{max}}}{r_{\text{min}}} n^{3/2}\right)^{n^2}
\end{align}
collisions. Here $m_{\text{max}}$ and $m_{\text{min}}$ denote the maximum and the minimum masses of the balls. Likewise, $r_{\text{max}}$ and $r_{\text{min}}$ denote the maximum and the minimum radii of the balls.
The following alternative upper bound for the maximum number of collisions appeared in \cite{BFK5}
\begin{align}\label{s26.2}
\left( 400 \frac{m_{\text{max}}}{m_{\text{min}}} 
 n^2\right)^{2n^4}.
\end{align}
  The papers \cite{BFK1,BFK2, BFK3,BFK4, BFK5} were the first to present universal bounds \eqref{s26.1}-\eqref{s26.2} on the number of collisions of $n$ hard balls in any dimension. 
The following  upper bound on the number of collisions has been given in \cite{uphard}, for balls with equal masses and radii,
\begin{align}\label{j19.2}
1600 \left( 1000 \cdot 32^{5^d}\right)^n  n^{((3/2)5^d+9/2) n +2} .
\end{align}
The new bound is better than those in \eqref{s26.1}-\eqref{s26.2} for a fixed dimension $d$ and large $n$.

It has been proved in \cite{BD} by example that the  number of elastic collisions of $n$ balls
in $d$-dimensional space is greater than $n^3/27$ for $n\geq 3$ and $d\geq 2$, for some initial conditions. The previously known lower bound was of order $n^2$ (that bound was for balls in dimension 1 and was totally elementary). The lower bound estimate was improved  in \cite{BuI18} to $2^{\lfloor n/2\rfloor}$ in dimensions $d\geq 3$.  

\subsection{Organization of the paper}

Section \ref{notdef} contains notation, definitions and the statement of the main theorem.
Section \ref{fold} is a short introduction to foldings. 
Much of the material in Sections \ref{se:intro}-\ref{fold} is copied from \cite{fold}, to help the reader follow and compare the two papers.
Section \ref{sec_fold} contains the main technical result---an upper bound for the size of an orbit of a point subjected to foldings. It also contains the proof of  Theorem \ref{au21.7}, our main result.
 Section \ref{pin_mov} 
addresses the question of whether pinned ball evolutions can be approximated by moving ball evolutions.
Section \ref{comp} contains a  comparison with 
 earlier results.
 
\section{Notation, definitions and the main result}\label{notdef}

Let $\calB(x, r)$ denote the closed ball with center $x$ and radius $r$ in a  Euclidean space (the dimension of which will be clear from the context).

We will consider a family of $n$ balls $\bF=\{\calB(x_1,1), \dots,\calB(x_n,1)\}$ in $\R^d$, for $d\geq 1$ and $n\geq 3$. We will assume that the interiors of the balls are disjoint but the balls may touch, i.e., for some pairs of balls, the distance between their centers is equal to 2.

We will say that $(\calV, \calE)$ is the \emph{full graph} associated with the family $\bF$ of $n$ pinned balls if   $\calV=\{x_1,x_2, \ldots , x_n\}$,  and   vertices $x_j$ and $x_k$ are connected by an edge if and only if $|x_j-x_k|=2$, i.e., if the balls $j$ and $k$ touch. An edge connecting $x_j$ and $x_k$ will be denoted $(j,k)$.

We will say that $(\calV_1, \calE_1)$ is a \emph{graph}  associated with the family $\bF$ if $(\calV_1, \calE_1)$ is a subgraph of the full graph $(\calV, \calE)$ associated with $\bF$, in the sense that $\calV_1 \subset \calV$ and $\calE_1 \subset \calE$.

Note that a full graph is not a complete graph in the graph-theoretic sense unless the  centers of the balls form the vertex set of a simplex.

We associate a pseudo-velocity $v_k\in \R^d$ to the $k$-th ball, for $k=1,\dots, n$. We call $v_k$ a pseudo-velocity because the balls do not move---their centers, i.e., $x_k$'s, are fixed. However, the pseudo-velocities will change due to pseudo-collisions as in an evolution of a family of billiard balls with totally elastic collisions. 
We will now define  mappings $\calT_{ij}: \R^{nd} \to \R^{nd}$ for  $1\leq i,j\leq n$. 
Let $ v  = (v_1, v_2,\dots, v_n)\in \R^{nd}$ and $\calT_{ij}( v ) =  w  =(w_1,\dots,w_n)$, with $w_k\in\R^d$ for every $k$, so that we can define $\calT_{ij}$ by specifying the values of $w_k$'s.

First, we let $w_k = v_k$ for every $k\ne i,j$. In other words, a collision between balls $i$ and $j$ does not affect the velocity of any other ball.

If the balls  $i$ and $j$ do not touch (i.e., $|x_i-x_j| > 2$) then we let $w_k = v_k$ for all $k$. Heuristically, balls which do not touch cannot collide.

If the balls  $i$ and $j$  touch and
\begin{align*}
(v_{i} - v_{j}) \cdot (x_{i} - x_{j}) \geq  0
\end{align*}
then once again there is no collision, i.e., we let $w_k = v_k$ for all $k$.

Finally, assume that the balls  $i$ and $j$  touch, i.e., $|x_i-x_j| = 2$, and
\begin{align*}
(v_{i} - v_{j}) \cdot (x_{i} - x_{j}) <  0.
\end{align*}
Let $u  = (x_{i} - x_{j})/|x_{i} - x_{j}|$.
Then we let
\begin{align}\label{a23.2}
w_{i} &= v_{i} + (v_{j} \cdot u) u
- (v_{i} \cdot u) u,\\
w_{j} &= v_{j} + (v_{i} \cdot u) u
- (v_{j} \cdot u) u.\label{a23.3}
\end{align}
In other words, the balls exchange the components of their pseudo-velocities that are parallel to the line through their centers; the orthogonal components  remain unchanged. This rule is identical to the classical totally elastic collision.

\begin{definition}
 The term ``collision'' 
will refer to the situation when $\calT_{ij}( v ) \ne v$, i.e., the (pseudo-)velocities change.
\end{definition}

Suppose that $(\calV_1, \calE_1)$ is a graph associated with a family  of pinned balls. Consider a sequence $\Gamma=(\gamma_j, j\geq 1)$, such that $\gamma_j \in \calE_1$ for every $j$. 
If $\gamma_j$ is an edge connecting vertices $x_i$ and $x_m$ then $\calT_{\gamma_j}$ should be interpreted as $\calT_{im}$.
For $v\in \R^{nd}$, let $v(0) = v$ and $v(j) = \calT_{\gamma_j}( v(j-1))$ for $j\geq 1$. 
The sequence $\Gamma$ represents the exogenous order of possible collisions for the system of pinned balls.

Note that  $v(j) = v(j-1)$ for some $\bF, v, \calE_1,\Gamma$ and $j$, i.e., a collision does not have to occur at every time $j$. 

Let $\tau_d$ denote
the  kissing number  of a $d$-dimensional ball, i.e.,  the maximum
number of mutually nonoverlapping translates of the ball that can be arranged
so that they all touch the ball.
According to  \cite[Thm.~1.1.3]{Bez}, 
\begin{align}\label{s11.1}
2^{0.2075d(1+ o(1))} \leq \tau_d \leq 2^{0.401d(1+o(1))}.
\end{align}
An elementary non-asymptotic bound is  
\begin{align}\label{f15.1}
2d\leq \tau_d \leq 3^d-1.
\end{align}

Our main result is the following upper bound for the number of collisions of pinned balls.

\begin{theorem}\label{au21.7}
For any given $\bF, v(0), \calE_1$ and $\Gamma$,
the number of collisions is bounded by
\begin{align}\label{f16.1}
 2^{-27}  2^{16\ell} n^{-27/2} n^{15\ell/2} .
\end{align}
where $n$ is the number of pinned balls, $d$ is the dimension of the space, and $\ell = \frac12 n \min \left( \tau_d ,\ n\right)$.
\end{theorem}

The proof of the theorem is given at the end of Section \ref{sec_fold}.

\section{Foldings}\label{fold}

We will  define foldings and show that collisions of a system of pinned balls can be represented as foldings.

Suppose that $h\in \R^d$ has unit length, and a closed half-space $H\subset \R^d$ is given by
\begin{align*}
H = \{v\in \R^{d } : \langle v , h\rangle \geq 0\}.
\end{align*}
It is clear that $\partial H = \left\{ v\in \R^d : \langle v , h\rangle = 0 \right\}$, in particular,  $0\in \partial H$. We define a \emph{folding} $\calF_H: \R^d\to \R^d$ relative to $H$ by
\begin{align*}
\calF_H(v) = 
\begin{cases}
v & \text{  if  } v\in H, \\
v - 2 \langle v , h\rangle h & \text{  if } v\notin H.
\end{cases}
\end{align*}
In other words, $\calF_H$ is the identity on $H$ and it is the reflection in the hyperplane $\prt H$ on the complement of $H$. Note that every folding is non-expansive, i.e., $$\dist\left(\calF_H(x) , \calF_H(y) \right) \leq \dist(x,y)$$ for all $x,y$ and $H$. 

\medskip

We will now represent transformation $\calT_{ij}$ as a folding.
For $k=1,\dots,n$ and $x\in \R^d$, let
\begin{align*}
x_{[k]} = ( \underbrace{0,\dots,0}_{(k-1)d}, x, \underbrace{0,\dots,0}_{(n-k)d} ) \in \R^{nd}.
\end{align*}
Let $\calE$ be the edge set for the full graph representing a family $\bF$. For $(j,k)\in \calE$, we let 
\begin{align}
\wt  z _{jk} &:= ( x_j - x_k)_{[j]} + (x_k - x_j)_{[k]} \in \R^{nd},
\label{j14.1}\\
  z _{jk} &:= \wt  z _{jk}/|\wt  z _{jk}| = 2^{-3/2}\wt  z _{jk} .
\label{au20.2}
\end{align}
Let $H_{jk}= \{ w  \in \R^{nd}:  \langle w , z _{jk}\rangle \geq 0\}$. The transformation $\calT_{ij}$ defined in \eqref{a23.2}-\eqref{a23.3} is the same as the folding  $\calF_{H_{i j}}$ in $\R^{nd}$, i.e.,
\begin{align}\label{au21.2}
\calT_{ij}=\calF_{H_{i j}}.
\end{align}
Note that $z_{ij} = z_{ji}$, hence $H_{ij} = H_{ji}$.

\medskip

Recall that $\calB(x, r)$ denotes the closed ball with center $x$ and radius $r$. The following result was proved in \cite{fold}.

\begin{theorem}\label{au20.1}
Suppose that $H_1, \dots, H_\ell$ are closed half-spaces in $\R^d$ such that their intersection $H_* :=\bigcap_{k=1}^\ell H_k$ has a non-empty interior. Assume that $(i_j)_{j\geq 1}$ is a sequence of integers  such that  $1 \leq i_j \leq \ell$ for all $j$. Suppose that $v_1\in \R^d$ and for $j> 1$, let $v_{j+1} = \calF_{H_{i_j}}  (v_{j})$.
Then there exists $k<\infty$ such that $v_j = v_k$ for all $j\geq k$. 
\end{theorem}

\begin{remark}\label{s28.6}
Let the set of all distinct elements of the sequence $(v_j, j\geq 1)$  be called an orbit. The orbit depends on  half-spaces $H_k$, sequence $(i_j, j\geq 1)$ and $v_1$. There is no upper bound on the size (cardinality) of the orbit that 
depends only on $n$, that is, the number of half-spaces $H_k$. To see this,
consider half-planes $H_1$ and $H_2$ defined in complex notation by $H_j=\{w \in \C: w \cdot e^{i \theta_j}\geq 0\}$, for $j=1,2$. It is easy to see that for any $m$, one can generate orbits with more than $m$ points by choosing $\theta_1$ and $\theta_2$ so that $|(\theta_1-\theta_2) - \pi|$ is non-zero but very small.
Hence, a universal bound for the orbit size depending only on $n$ does not exist. 

\end{remark}

\section{A bound on the number of foldings}\label{sec_fold}

Remark \ref{s28.6} shows that we need an extra assumption to find a bound for the orbit size. The following definition presents such an assumption. The assumption is chosen to make the bound applicable to counting of pinned ball collisions. 

\begin{definition}\label{f10.1}
Suppose $\{H_k\}_{1\leq k\leq \ell}$ are closed half-spaces in $\R^d$, for some $d\geq 2$ and $\ell\geq 1$.  
Suppose that $v_1, w_0 \in \R^d$, $0<r<1$, $\calB(w_0, r) \subset \bigcap_{k=1}^\ell H_k$, and $|w_0 - v_1| \leq 1$.
Assume that   $\{\alpha_j\}_{j\geq 1}$ is a sequence of integers such that $1\leq \alpha_j \leq \ell$ for all $j$. Let $v_{k+1} = \calF_{H_{\alpha_k}}(v_k)$ for $k\geq 1$.

(i) We will say that a jump occurred at time $k$ if $v_{k+1} \ne v_k$.

(ii) Given $\ell\geq 1$ and $0<r<1$, the maximum number of jumps in the sequence $\{v_k\}_{k\geq 1}$
will be denoted $N(\ell,r)$, where the maximum is taken over all $d\geq 2$, $\{H_k\}_{1\leq k\leq \ell}$, $v_1, w_0 \in \R^d$, $\calB(w_0, r)$ and   $\{\alpha_j\}_{j\geq 1}$ satisfying the conditions stated at the beginning of the definition.

\end{definition}

\begin{lemma}\label{f11.1}
If $\ell\geq 3$ and $0<r<1$ then $N(\ell,r) \leq 308 r ^{-5}N(\ell-1,r)$.

\end{lemma}

\begin{proof}
Suppose $\{H_k\}_{1\leq k\leq \ell}$ are closed half-spaces in $\R^d$,   $v_1, w_0 \in \R^d$, 
\begin{align}\label{f20.1}
0<r<1,\qquad \calB(w_0, r) \subset \bigcap_{k=1}^\ell H_k,
\qquad |w_0 - v_1| \leq 1.
\end{align}
Assume that   $\{\alpha_j\}_{j\geq 1}$ is a sequence of integers such that $1\leq \alpha_j \leq \ell$ for all $j$. Let $v_{k+1} = \calF_{H_{\alpha_k}}(v_k)$ for $k\geq 1$.
Let $\eta \in \left( \sqrt{8}-2, 1 \right)$ and define
\begin{align}\label{f20.2}
\eps = \dist\left(v_1,\bigcap_{k=1}^\ell H_k\right),
\qquad \delta = 4\eps r^2 \left(\frac{1-\eta}{4+3\eta r}\right),
\qquad \rho = r-\delta .
\end{align}
If $\eps = 0$ then the number of jumps is zero. From now on we will assume that $\eps>0$. It follows from \eqref{f20.1} that $\eps\leq 1$.

It is elementary to check that for $r>0$ and $0< \eps \leq 1$, it holds that $\delta\leq \eps r^2 (1-\eta) < r(1-\eta)$. It follows that $\rho > \eta r$. Therefore,
\begin{align*}
\frac{\rho \eps}{2} - \frac{3\delta}{2} - \frac{2\delta}{\rho} \geq \frac{\eta r \eps}{2} - \frac{3\delta}{2} - \frac{2\delta}{\eta r} =  \frac{\eta r \eps}{2} - \left( \frac{3\eta r+4}{2\eta r} \right) \delta.
\end{align*}

We combine this with the definition \eqref{f20.2} of $\delta$ and the assumption that $\eta \in \left( \sqrt{8}-2, 1 \right)$ to obtain
\begin{align}\label{f9.5}
\frac{\rho \eps}{2} - \frac{3\delta}{2} - \frac{2\delta}{\rho} \geq \left( \frac{(\eta+2)^2-8}{2\eta}\right) r\eps > 0.
\end{align}

\medskip
\emph{Step 1}.
In this step we will show that for every $j\geq 1$ there exists $k$ such that $\dist(v_j, \prt H_k) > \delta$. Suppose otherwise, i.e., there exists $j_*\geq 1$ such that if we let $w_1=v_{j_*}$ then for all $k$,
\begin{align}\label{f4.2}
\dist(w_1, \prt H_k)=\dist(v_{j_*}, \prt H_k) \leq \delta.
\end{align}
We will prove that this assumption leads to a contradiction.

\bigskip

We start by establishing three preliminary inequalities.

Suppose that $w_1 \notin \bigcap_{k=1}^\ell H_k$ and, therefore, there is $k$ such that $w_1 \notin H_{k}$. Consider any $k$ such that $w_1 \notin H_{k}$.
Let $L$ be the line passing through $w_0$ and $w_1$.
Let $a_k \in L$ be such that  $\ol{w_0,a_k}= \ol{w_0,w_1} \cap H_{k}$. Since 
$|w_0-w_1| \leq 1$, $\dist(w_0, \prt H_{k}) \geq r$, and $\dist(w_1, \prt H_{k})\leq \delta$, elementary geometry shows that $|w_1 - a_k|\leq \delta/(r+\delta)$. It is easy to see that there is a unique point $a_*$ such that  $\{a_*\}= L\cap \prt \left( \bigcap_{k=1}^\ell H_k\right)$. There exists $k_1$ such that $\{a_*\}= L\cap \prt H_{k_1}$ and $w_1 \notin H_{k_1}$.
Then
\begin{align*}
\dist\left(w_1, \bigcap_{k=1}^\ell H_k\right)
\leq \dist(w_1,a_*) = |w_1 - a_{k_1}|\leq \delta/(r+\delta).
\end{align*}
The bound also holds in the case when $w_1 \in \bigcap_{k=1}^\ell H_k$. Hence,
\begin{align}
\label{j11.1}
|v_1 - w_1| \geq  \dist\left(v_1, \bigcap_{k=1}^\ell H_k\right)
- \dist\left(w_1, \bigcap_{k=1}^\ell H_k\right)
\geq \eps - \delta/(r+\delta).
\end{align}

Let $\be = (w_1-w_0)/|w_1-w_0|$ and let $\n_k$ be the unit inward normal vector to $\prt H_k$. 
Since $\rho= r-\delta \leq (r-\delta)/|w_0 - w_1|$, $\calB(w_0, r) \subset \bigcap_{k=1}^\ell H_k$, and $\dist(w_1, \prt H_k) \leq \delta$ for all $k$, elementary geometry shows that  for all $k$,
\begin{align}\label{f4.1}
\langle \n_k, \be\rangle \leq - \rho.
\end{align}

Let
$w_2 = w_1 + (2\delta/\rho) \be$. It follows from the fact that $|\langle \n_k, w_1\rangle |= \dist(w_1, \prt H_k)$ and bounds  \eqref{f4.2} and \eqref{f4.1},  that $w_2 \in \bigcap_{k=1}^\ell H_k^c$.
In view of \eqref{j11.1},
\begin{align}\label{f9.3}
|v_1 - w_2|\geq |v_1 - w_1| - |w_1-w_2| \geq\eps- \delta/(r+\delta)-2\delta/\rho\geq\eps- 3\delta/\rho.
\end{align}

\medskip

We will use the inequalities established above to estimate the distance between orbit points $v_k$ and $w_1$.
If $\alpha_k = j$ and $v_k\notin H_j$ then $v_{k+1} = v_k + \n_j | v_{k+1} - v_k |$ and, by \eqref{f4.1},
\begin{align}\label{f4.3}
\langle v_{k+1}, \be\rangle &= \langle v_{k}, \be\rangle + \langle \n_j, \be\rangle | v_{k+1} - v_k |
\leq \langle v_{k}, \be\rangle -\rho| v_{k+1} - v_k |.
\end{align}
The estimate holds also in the case when $v_k\in H_j$, i.e., when $v_{k+1} = v_k $. By \eqref{f4.3} and the triangle inequality, for any $m>k$,
\begin{align}\label{f4.4}
\langle v_{m}, \be\rangle 
- \langle v_{k}, \be\rangle
= \sum_{j=k}^{m-1}  \langle v_{j+1}, \be\rangle 
- \langle v_{j}, \be\rangle
\leq -\sum_{j=k}^{m-1} \rho |v_{j+1} - v_j|
 \leq -\rho|v_{m} - v_k|.
\end{align}

Consider any $1\leq m\leq j$ and suppose that $\langle \be, v_m\rangle \leq \langle \be, w_2\rangle $.
Let $a= |v_m - w_2|$.
 If $|v_j-v_m| \leq a/2$ then, by the triangle inequality, $|v_j - w_2|\geq a/2$.
If $|v_j-v_m| \geq a/2$ then, by \eqref{f4.4},
\begin{align*}
\langle v_{j}, \be\rangle 
- \langle w_2, \be\rangle
\leq
\langle v_{j}, \be\rangle 
- \langle v_{m}, \be\rangle
\leq - \rho|v_{j} - v_m| \leq -\rho a/2. 
\end{align*}  
Since $\rho\leq r \leq 1$, we obtain
\begin{align}\label{f8.3}
|v_j - w_2|\geq 
\min(a/2, \rho a/2)=
\rho a/2 = \rho |v_m - w_2| /2,
\end{align}
 for all $j\geq m$ if 
$\langle \be, v_m\rangle \leq \langle \be, w_2\rangle $.

Assume that $\langle \be, v_1\rangle \leq \langle \be, w_2\rangle $.
Then  by \eqref{f9.3} and \eqref{f8.3}, for all $j\geq 1$,
\begin{align}\label{f9.1}
|v_j - w_2|\geq \rho|v_1 - w_2|/2 \geq \rho(\eps-3\delta/\rho)/2.
\end{align}

Next suppose that $\langle \be, v_1\rangle \geq \langle \be, w_2\rangle $.
We will use induction. Suppose that $\langle \be, v_j\rangle \geq \langle \be, w_2\rangle $ and $|v_j - w_2|\geq \eps-3\delta/\rho$. Note that by \eqref{f9.3} and the current assumption, these conditions hold for $j=1$.

Suppose that $v_{j+1} = \calF_{H_k}(v_j)$.
Recall that $L$ is the line passing through $w_0$ and $w_1$, and $a_k$ is the intersection point of $L$ and $\prt H_k$. Then $|v_{j+1}-a_k| = |v_{j}-a_k|$, and thus
\begin{align*}
|v_{j+1}-w_2|^2 &= |v_{j+1}-a_k|^2 + 2 \langle v_{j+1} -a_k, a_k-w_2\rangle + |a_k-w_2|^2 \\
&= |v_{j}-a_k|^2 + 2 \langle v_{j+1} -a_k, a_k-w_2\rangle + |a_k-w_2|^2 \\
&= |v_{j}-w_2|^2 + 2 \langle v_{j+1} - v_j, a_k-w_2\rangle.
\end{align*}
The points $w_2, a_k$, and $w_0$ are colinear, and since $w_2 \in \bigcap_{k=1}^\ell H_k^c$, it follows that $\langle \be, a_k\rangle \leq \langle \be, w_2\rangle $. Hence, $a_k-w_2 = -|a_k-w_2|\be$. Combining this with the above formula we obtain  
\begin{align*}
|v_{j+1}-w_2|^2 &=  |v_{j}-w_2|^2 + 2 |a_k-w_2| \langle v_j - v_{j+1} , \be\rangle.
\end{align*}
The inner product on the right hand side is positive by \eqref{f4.3}. Using the induction assumption, we conclude that
\begin{align*}
|v_{j+1} - w_2|\geq|v_j - w_2|\geq \eps-3\delta/\rho.
\end{align*}
It follows by induction that if $\langle \be, v_i\rangle \geq \langle \be, w_2\rangle $ for all $i=1,\dots,j$ then $|v_{j+1} - w_2|\geq \eps-3\delta/\rho$.

Let $i$ be the largest index such that  $\langle \be, v_i\rangle \geq \langle \be, w_2\rangle $. Then for $j=1,2,\dots, i+1$,
\begin{align}\label{f9.2}
|v_{j} - w_2|\geq  \eps-3\delta/\rho.
\end{align}
By the definition of $i$, $\langle \be, v_{i+1}\rangle \leq \langle \be, w_2\rangle $. Therefore, by \eqref{f8.3} and \eqref{f9.2}, for all $j\geq i+1$,
\begin{align}\label{f9.4}
|v_j - w_2|\geq  \rho |v_{i+1} - w_2| /2
\geq  \rho (\eps-3\delta/\rho) /2.
\end{align}
It follows from \eqref{f9.1}, \eqref{f9.2} and \eqref{f9.4} that for all $j\geq 1$,
\begin{align*}
|v_j - w_2|\geq   \rho (\eps-3\delta/\rho) /2.
\end{align*}
Hence, for all $j\geq 1$, by \eqref{f9.5},
\begin{align*}
|v_j - v_{j_*}|&=|v_j - w_1| \geq 
|v_j - w_2|-|w_1-w_2|\geq   \rho (\eps-3\delta/\rho) /2 -2\delta/\rho\\
&= \frac{\rho \eps}{2} - \frac{3\delta}{2} - \frac{2\delta}{\rho} \geq \left( \frac{(\eta+2)^2-8}{2\eta}\right) r\eps > 0.
\end{align*}
This is a contradiction, because we can take $j=j_*$. This completes the proof that for every $j\geq 1$ there exists $k$ such that $\dist(v_j, \prt H_k) \geq \delta$.

\medskip
\emph{Step 2}.
Consider $v_i, v_j$ for some $i<j$ and assume that $|v_i-w_0|=R>r$ and $|v_i-v_j|\geq \delta_1>0$. We will estimate the difference $|v_j-w_0|-|v_i-w_0|$. Consider the two dimensional plane holding $v_i, v_j$ and $w_0$, illustrated in Fig. \ref{fig6}. The points are arranged in our illustration so that $v_i$ and $w_0$ lie on a vertical line, $v_i$ is above $w_0$, and $v_j$ is below $v_i$ or to the left of the line passing through $v_i$ and $w_0$. 
Let $L$ denote the line passing through $w_0$ and orthogonal to $\ol{v_i w_0}$.
Let $z_1\in L$ be  such that $|z_1-w_0| = r$, and $z_1$ and $v_j$ lie on the opposite sides of $\ol{v_i w_0}$.

\begin{figure} \includegraphics[width=0.6\linewidth]{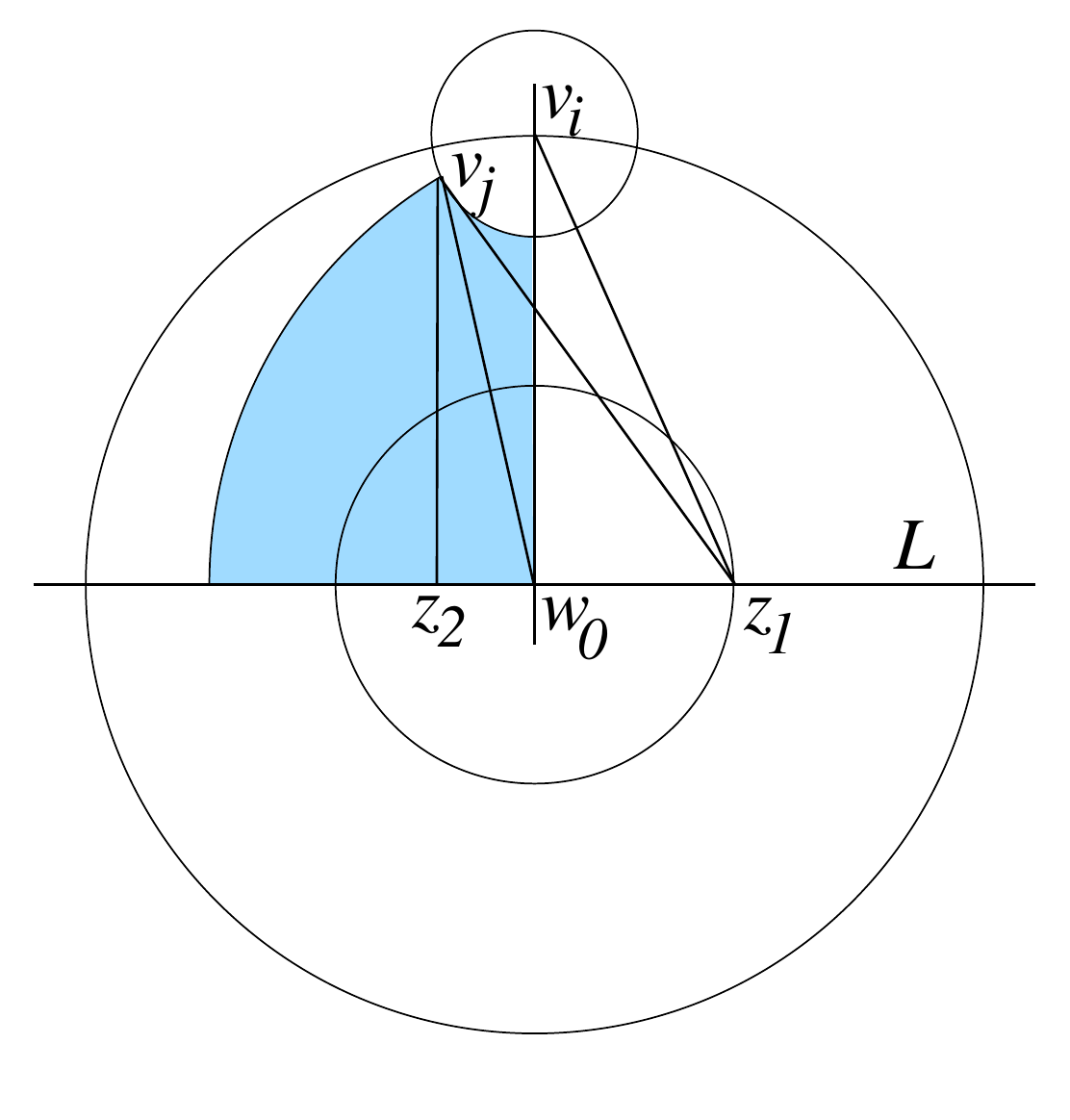}
\caption{The point $v_j$ must be located in the blue region if it is located above $L$.}
\label{fig6}
\end{figure}

We will consider two cases---when $v_j$ lies above $L$ or below $L$. Fig. \ref{fig6} illustrates the case when $v_j$ lies above $L$.
Since $z_1\in\calB(w_0,r)$, $z_1$ belongs to all half-spaces $H_k$. Foldings do not increase distance so we must have $|v_j-z_1| \leq |v_i - z_1|$. The blue region in Fig. \ref{fig6} is the closure of the intersection of the following sets: $\calB(z_1, |v_i-z_1|)$, half-plane to the left of $\ol{v_i w_0}$, half-plane above $L$, and the exterior of $\calB(v_i, \delta_1)$. The point $v_j$ must lie in the blue region. It is elementary to see that among all points in the blue region, the largest distance from $w_0$ is attained at the intersection of the circles $\prt \calB(z_1, |v_i-z_1|)$ and $\prt \calB(v_i, \delta_1)$. This is where we placed $v_j$ in Fig. \ref{fig6}. Our estimate will apply to all other possible positions of $v_j$ in the blue region.

Let $z_2$ be the orthogonal projection of $v_j$ onto $L$.
Let
\begin{align*}
\alpha &= \angle (\ol{v_i z_1}, \ol{v_jz_1}),\\
\beta &= \angle (\ol{v_i z_1}, \ol{w_0 z_1}),\\
\gamma &= \angle (\ol{v_j z_1}, \ol{w_0 z_1})= \beta - \alpha,\\
a&= |v_i-z_1|= |v_j-z_1| = \sqrt{r^2+R^2}.
\end{align*}
We have
\begin{align*}
&\sin(\alpha/2) = (\delta_1/2)/ a,\\
& \cos(\alpha/2) = \sqrt{1 - \sin^2(\alpha/2)}
= \sqrt{1 - \delta_1^2 /(4a^2)},\\
& \sin \alpha = 2 \sin(\alpha/2) \cos(\alpha/2)
= \frac {\delta_1} a \sqrt{1 - \delta_1^2 /(4a^2)},\\
& \cos \alpha = \sqrt{1 - \sin^2\alpha}
= 1- \frac{\delta_1^2}{2 a^2},\\
&\cos \beta = r/a,\\
&\sin \beta = \sqrt{1 - \sin^2 \beta} = \sqrt{1 - r^2 /a^2},\\
&\cos \gamma = \cos (\beta - \alpha)
= \cos \alpha \cos \beta + \sin \alpha \sin \beta\\
&\qquad = \left(1- \frac{\delta_1^2}{2 a^2}\right) \frac r a 
+ \frac {\delta_1} a \sqrt{1 - \delta_1^2 /(4a^2)} 
\sqrt{1 - r^2 /a^2},\\
&|v_j - z_2| = |v_j - z_1| \sin \gamma = a \sin \gamma,\\
& |z_2 - w_0| = |z_2 - z_1| - r = |v_j - z_1| \cos \gamma -r
= a \cos \gamma -r,\\
&|v_j-w_0| ^2 = |v_j - z_2|^2 + |z_2 - w_0|^2
= a^2 \sin^2 \gamma + (a \cos \gamma -r)^2\\
&\qquad = a^2 - 2 a r \cos\gamma + r^2\\
&\qquad =a^2 + r^2 - 2 a r\left[\left(1- \frac{\delta_1^2}{2 a^2}\right) \frac r a 
+ \frac {\delta_1} a \sqrt{1 - \delta_1^2 /(4a^2)} 
\sqrt{1 - r^2 /a^2}\right]\\
&\qquad =a^2 + r^2 -2r^2 +  \frac{2r^2\delta_1^2}{2 a^2}
-2 r \delta_1 \sqrt{1 - \delta_1^2 /(4a^2)} 
\sqrt{1 - r^2 /a^2}\\
&\qquad = a^2 - r^2 +  \frac{r^2\delta_1^2}{ a^2}
- \frac {r\delta_1}{a^2}\sqrt{4a^2 - \delta_1^2 } 
\sqrt{a^2 - r^2 }\\
&\qquad = R^2+r^2 - r^2 +  \frac{r^2\delta_1^2}{ R^2+r^2}
- \frac {r\delta_1}{R^2+r^2}\sqrt{4(R^2+r^2) - \delta_1^2 } 
\sqrt{R^2+r^2 - r^2 }\\
&\qquad = R^2 +  \frac{r^2\delta_1^2}{ R^2+r^2}
- \frac {rR\delta_1}{R^2+r^2}\sqrt{4(R^2+r^2) - \delta_1^2 }.
\end{align*}
Note that $\delta_1 \leq 2 R$ and, therefore, $\sqrt{4(R^2+r^2) - \delta_1^2 } \geq 2 r$. Hence,
\begin{align*}
&|v_j-w_0| ^2 \leq
R^2 +  \frac{r^2\delta_1^2}{ R^2+r^2}
- \frac {2r^2R\delta_1}{R^2+r^2}
= R^2 -\delta_1   \frac{r^2}{ R^2+r^2}
(2R -\delta_1)\\
&\qquad \leq R^2 -\delta_1  r^2
(2R -\delta_1)/2.
\end{align*}
If $v_j$ is located in the blue region then $\delta_1 \leq \sqrt{2} R$ and, therefore,
\begin{align}\notag
&|v_j-w_0| ^2 \leq R^2 -\delta_1  r^2
(2R -\sqrt{2} R)/2 \leq R^2 -\delta_1  r^3
(2 -\sqrt{2} )/2 ,\\
&|v_j-w_0|  \leq\sqrt{R^2 -\delta_1  r^3
(2 -\sqrt{2} )/2 }
= R\sqrt{1 -\delta_1  r^3
(2 -\sqrt{2} )/(2R^2) }\notag\\
&\qquad \leq  R(1 -\delta_1  r^3
(2 -\sqrt{2} )/(4R^2) ) = R -\delta_1  r^3
(2 -\sqrt{2} )/(4R)\notag\\
&\qquad \leq |v_i-w_0| -\delta_1  r^3/8. \label{m20.1}
\end{align}

Next consider the case when $v_j$ is below the line $L$, as illustrated in Fig. \ref{fig7}. Let $z_3$ lie on the line segment between $v_i$ and $w_0$, and be such that $|z_3 - w_0|=r$.
\begin{figure} \includegraphics[width=0.6\linewidth]{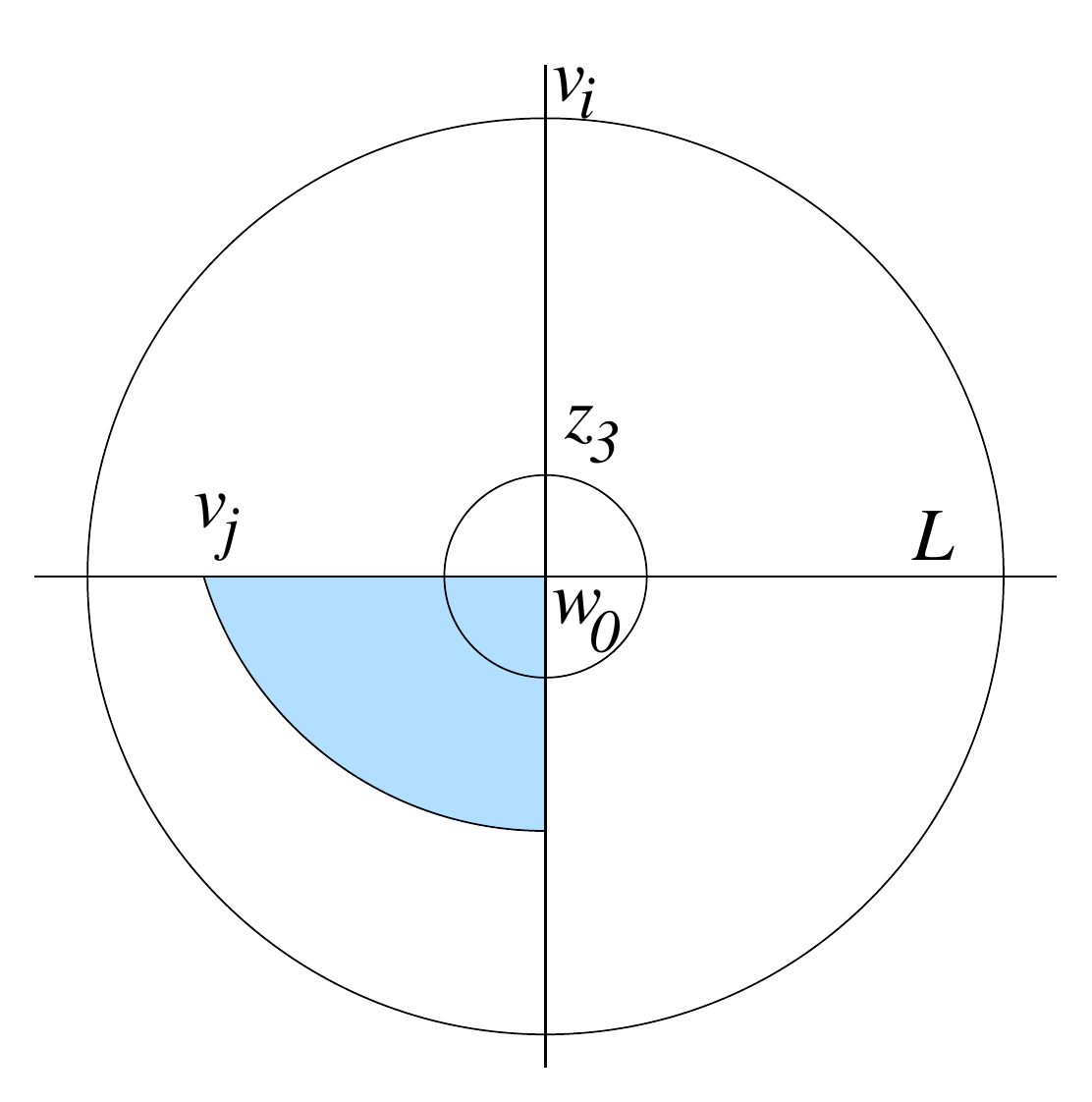}
\caption{The point $v_j$ must be located in the blue region if it is located below $L$.}
\label{fig7}
\end{figure}
Since $z_3\in\calB(w_0,r)$, $z_3$ belongs to all half-spaces $H_k$. Foldings do not increase distance so we must have $|v_j-z_3| \leq |v_i - z_3|$. The blue region in Fig. \ref{fig7} is the intersection of the following sets: $\calB(z_3, |v_i-z_3|)$, half-plane to the left of $\ol{v_i w_0}$ and half-plane below $L$. Among all points in the blue region, the largest distance from $w_0$ is attained at the intersection of  $\prt \calB(z_3, |v_i-z_3|)$ and $L$. This is where we placed $v_j$ in Fig. \ref{fig7}. Our estimate will apply to all other possible positions of $v_j$ in the blue region.
\begin{align*}
&|v_j - w_0|^2 = |v_j-z_3|^2 - |z_3 - w_0|^2
= |v_i-z_3|^2 - r^2
= (|v_i-w_0|-r)^2 - r^2\\
&\qquad =
|v_i-w_0|^2 - 2 r |v_i-w_0|\leq |v_i-w_0|^2 - 2 r^2,\\
&|v_j - w_0| \leq \sqrt{|v_i-w_0|^2 - 2 r^2}
= |v_i-w_0|\sqrt{1 - 2 r^2/|v_i-w_0|^2}\\
&\qquad \leq |v_i-w_0| (1 -  r^2/|v_i-w_0|^2)
= |v_i-w_0|  -  r^2/|v_i-w_0|
\leq |v_i-w_0|  -  r^2.
\end{align*}
Since $\delta_1 \leq 2$ and $r\leq 1$, the last estimate and \eqref{m20.1} imply that 
\begin{align}\label{m20.2}
&|v_j - w_0| \leq  |v_i-w_0|  - \min(\delta_1  r^3/8,  r^2)
\leq  |v_i-w_0|  -\delta_1  r^3/8.
\end{align}

\medskip

\emph{Step 3}.
We will use an inductive argument, by recursively constructing three sequences. Let $j_1=1$ and let $k_1$ be the smallest integer such that $\dist(v_1, \prt H_{k_1}) \geq \delta$. Let $m_1 \in [1,\infty]$ be the minimum number such that $\alpha_{m_1} = k_1$ and $v_{m_1+1}\ne v_{m_1}$. 
We interpret $m_1=\infty$ as non-existence of $m_1$ satisfying these conditions.
The number of jumps at times $ 1, \dots, m_1$ is bounded by $N(\ell-1,r)+1$. 
Note that  the bound applies also  when $m_1=1$ or $m_1=\infty$.

Suppose that $m_1<\infty$. Then $v_{m_1}\notin H_{k_1}$ and $v_{m_1+1}\in H_{k_1}$.
If $v_1 \notin H_{k_1}$ then $|v_{m_1+1}-v_1|\geq \delta$ and  by \eqref{m20.2},
\begin{align}\label{f1.1}
|v_{m_1+1}-w_0| \leq |v_1-w_0|- \delta  r^3/8.
\end{align}
If $v_1 \in H_{k_1}$ then  $|v_{m_1}-v_1|\geq \delta$ and  by \eqref{m20.2},
\begin{align}\label{f9.6}
|v_{m_1+1}-w_0| \leq|v_{m_1}-w_0| \leq |v_1-w_0|- \delta  r^3/8.
\end{align}

If $j_i, k_i$, and $m_i$ have been defined, $j_i<\infty$  and $m_i < \infty$, 
then we let $j_{i+1}= m_i+1$. Otherwise $j_{i+1} =\infty$. If $j_{i+1} <\infty$, let $k_{i+1}$ be the smallest integer  such that $\dist(v_{j_{i+1}}, \prt H_{k_{i+1}}) \geq \delta$.
Let $m_{i+1} \in [j_{j+1},\infty]$ be the minimum number such that $\alpha_{m_{i+1}} = k_{i+1}$ and $v_{m_{i+1}+1}\ne v_{m_{i+1}}$. 
We interpret $m_{i+1}=\infty$ as non-existence of $m_{i+1}$ satisfying these conditions.
The number of jumps at times $ j_{i+1}, \dots, m_{i+1}$ is bounded by $N(\ell-1,r)+1$. 
The bound applies also in the case when $m_{i+1}=j_{i+1}$ or $m_{i+1}=\infty$.
If $m_{i+1}<\infty$ then $v_{m_{i+1}}\notin H_{k_1}$ and $v_{m_{i+1}+1}\in H_{k_1}$. The same argument that yields \eqref{f1.1} and \eqref{f9.6} gives
\begin{align*}
|v_{m_{i+1}+1}-w_0| \leq |v_{j_{i+1}}-w_0|- \delta  r^3/8.
\end{align*}
This implies that for all $i\geq 0$ such that $j_{i+2} < \infty$,
\begin{align}\label{f23.5}
|v_{j_{i+2}}-w_0|=
|v_{m_{i+1}+1}-w_0| \leq |v_{j_{i+1}}-w_0|- \delta  r^3/8.
\end{align}

Recall definition \eqref{f20.2} and let $y_1 \in \bigcap_{k=1}^\ell H_k$ be such that $\dist(v_1, y_1) = \eps$. The point $y_1$ is invariant under all foldings, and foldings do not increase distance between points so $\dist(v_j, y_1) = \eps$ for all $j$. Hence for all $j\geq 1$, 
\begin{align*}
0 \leq |v_{1}-w_0| - |v_{j}-w_0|  \leq |v_1 - v_j| 
\leq \dist(v_1, y_1) +\dist(v_j, y_1) \leq 2\eps.
\end{align*}
This and \eqref{f23.5} imply that 
if $j_i<\infty$  then  
\begin{align}\label{f2.1}
-2\eps&\leq 
|v_{j_{i}}-w_0|-|v_{1}-w_0|
= \sum_{k=1}^{i-1}\left(
|v_{j_{k+1}}-w_0|-|v_{j_k}-w_0|
\right)\\
&\leq \sum_{k=1}^{i-1} -\delta  r^3/8
= -(i-1) \delta  r^3/8.\notag
\end{align}
We  see that if $j_i<\infty$  then
$(i-1) \delta  r^3/8 \leq 2\eps$, which is equivalent to 
$i \leq 16 \eps/ (r^3 \delta) +1$.
Hence, $m_i = \infty$  for some $i \leq 16 \eps/ (r^3 \delta) +2$.
Recall that the number of jumps at times $ j_{i+1}, \dots, m_{i+1}$ is bounded by $N(\ell-1,r)+1$. 
Hence, using \eqref{f20.2} we can bound   the number of jumps  by 
\begin{align*}
&(16 \eps/ (r^3 \delta) +2)(N(\ell-1,r)+1)
= \left(\frac{4}{r^5} \left( \frac{4+3\eta r}{1-\eta} \right)  +2\right)(N(\ell-1,r)+1).
\end{align*}
This expression  is an upper bound for the number of jumps for any value of $\eta \in \left( \sqrt{8}-2, 1 \right)$. The quotient  $ \frac{4+3\eta r}{1-\eta} $ is an increasing function of $\eta$ for every fixed $r \in (0,1)$. Thus, we can optimize the bound by taking the limit as $\eta \downarrow\sqrt{8}-2$. It follows that the number of jumps is bounded by
\begin{align*}
&\left(\frac{4}{r^5} \left( \frac{4+3(\sqrt{8}-2) r}{3-\sqrt{8}} \right)  +2\right)  (N(\ell-1,r)+1) \\
&\quad\leq \left(\frac{4}{r^5} \left( \frac{3\sqrt{8}-2}{3-\sqrt{8}} \right)  +2\right)(N(\ell-1,r)+1) 
\leq \left( \frac{10\sqrt{8}-2}{r^5(3-\sqrt{8})} \right)  (N(\ell-1,r)+1) \\
&\quad\leq \frac{154}{r^5}  (N(\ell-1,r)+1) \leq \frac{308}{r^5}  N(\ell-1,r).
\end{align*}
\end{proof}

\begin{lemma}\label{f11.2}
If  $0<r<1$ then
$N(2,r) \leq \pi/r+1$.
\end{lemma}

\begin{proof}
It is easy to see that it will suffice to consider  $d= 2$.
Consider two half-planes $H_1$ and $H_2$ in $\R^2$.
Suppose that there exist $v_1, w_0 \in \R^d$, $0<r<1$, such that $\calB(w_0, r) \subset  H_1 \cap H_2$ and $|w_0 - v_1| \leq 1$.

Assume that $\prt H_1$ and $\prt H_2$ intersect at a (finite) point $z$ and
let $\alpha\in(0,\pi]$ be the angle of the wedge $H_1 \cap H_2$. 
If $v_1=z$ then there will be no jumps. Assume that $v_1\ne z$.

Let $L$ be the half-line bisecting $H_1 \cap H_2$. Let $\theta_k\in[0,\pi]$ be the angle between $L$ and the half-line starting at $z$ and passing through $v_k$. If $\theta_k\leq \alpha/2$ then $\theta_{k+1}= \theta_k$. 

Suppose that $\theta_1>\alpha/2$. 
It is easy to see that $\theta_2 \leq \theta_1$. Then $v_k \in H_1$ or $v_k \in H_2$, for every $k\geq 2$. This easily implies that  $\theta_{k+1}= \theta_k-\alpha$ for $k\geq 2$. Hence the number of jumps is bounded by $\theta_1/\alpha +1$.
We have
\begin{align*}
\alpha &\geq \arcsin\left(\frac{r/2}{|z-w_0|}\right)
\geq \frac{r/2}{|z-w_0|},\\
\theta_1 &\leq \arcsin\left(\frac{1}{|z-w_0|}\right)
\leq \frac{\pi/2}{|z-w_0|},\\
\theta_1/\alpha +1 & \leq \frac{\pi/2}{|z-w_0|} \cdot \frac{|z-w_0|}{r/2}+1
=\pi/r+1.
\end{align*}
It is easy to check that the upper bound $\pi/r+1$ applies also in the case when $\prt H_1$ and $\prt H_2$ do not intersect.
\end{proof}

\begin{theorem}\label{f11.3}
If $\ell\geq 2$ and $0<r<1$ then
\begin{align*}
N(\ell,r) \leq  2\pi\,  308^{\ell-2} r^{-5(\ell-2)-1}.
\end{align*}
\end{theorem}

\begin{proof}
By Lemmas \ref{f11.1} and \ref{f11.2}, and induction,
\begin{align*}
N(2,r) &\leq \pi/r+1 \leq 2\pi/r,\\
N(\ell,r) & \leq 2\pi r^{-1} (308 r^{-5}) ^{\ell-2}
=2\pi \, 308^{\ell-2} r^{-5(\ell-2)-1}.
\end{align*}
\end{proof}

\begin{proof}[Proof of Theorem \ref{au21.7}]
We will assume that the full graph associated with $\bF$ is connected.
Otherwise, we can move connected  components so that they touch each other. This can only increase the number of possible collisions.

Let $ w  = (w_1, w_2, \dots, w_n)$ where $w_k = c (x_k-x_1)$ for $k=1,\dots,n$ and $c>0$ is chosen  so that $\sum_{k=1}^n |w_k|^2 = 1$, i.e., $w\in \prt \calB(0,1)$.
Recall that we have assumed that the full graph associated with the family of pinned balls is connected. This implies that $|x_k-x_1| \leq 2 (n-1)$ for every $k$, and, therefore,  $\sum_{k=1}^n |x_k-x_1|^2 \leq 4 n (n-1)^2 $. Hence, $c^2\geq 1/(4n(n-1)^2)$.
We have
\begin{align}\label{au6.1}
 w  \cdot  z _{ij} &= 
(w_i - w_j) \cdot (x_i - x_j)
/ |\wt  z  _{ij}|
= 2^{-3/2}
(w_i - w_j) \cdot (x_i - x_j)\\
&= 2^{-3/2}c(x_i-x_1- x_j+x_1) \cdot (x_i - x_j)
 =2^{-3/2} c |x_i - x_j|^2 \nonumber \\
&\geq 2^{-3/2}\cdot \frac 1 {2\sqrt{n}(n-1)}\cdot 2^2
= \frac 1 {\sqrt{2n}(n-1)}.\nonumber
\end{align}
This shows that 
if $r= 1/(\sqrt{2n}(n-1))$ then
$\calB(w, r) \subset \bigcap H_{ij} $, where the intersection extends over all half-spaces $H_{ij}$ defined in Section \ref{fold}.
In Definition \ref{f10.1}, we take $w_0$ equal to $w$ defined in this proof, and we take $v_1$ in that definition to be the vector of the initial pseudo-velocities. Since $w, v_1\in \prt \calB(0,1)$, we have $|w-v_1|\leq 2$. Hence, if we use dilation with scaling factor 2, the effective $r$ is twice as small as $r= 1/(\sqrt{2n}(n-1))$, i.e., we should take $r= 2^{-3/2}/(\sqrt{n}(n-1))$.

If the number of balls is $n$ then the number  of pairs of touching balls is bounded by $\tau_d n /2$; it is also bounded by  $n(n-1)/2 \leq n^2/2$.
Hence, Theorem \ref{f11.3} can be applied with  $\ell=\tau_d n /2$ or $\ell=  n^2/2$.

By the identification of collisions and foldings given in \eqref{au21.2}, and Theorem \ref{f11.3}, the number of collisions is bounded by
\begin{align*}
&2\pi\,  308^{\ell-2} \left( \frac{2^{-3/2}}{\sqrt{n}(n-1)} \right)^{-5(\ell-2)-1}
\leq 2\pi\ 308^{\ell-2}\ 2^{15(\ell-2)/2+3/2} \ n^{15(\ell-2)/2+3/2}  \\
&= 2\pi\ 308^{-2}\ 2^{-15+3/2}\ 308^{\ell}\ 2^{15\ell/2} \ n^{15\ell/2-27/2}
\leq 2^{-27}  2^{16\ell} n^{-27/2} n^{15\ell/2}.
\end{align*}
The bound in \eqref{f16.1} follows from the fact that the above expression is increasing in $\ell$, and we can apply it with the two possible values for $\ell$, i.e., $\ell=\tau_d n /2$ or $\ell=  n^2/2$.
\end{proof}

\section{Pinned and moving billiard balls}\label{pin_mov}

One may ask whether an upper bound for the number of collisions of pinned balls  could be derived from the upper bounds \eqref{s26.1}-\eqref{j19.2} for the number of collisions of moving billiard balls. A possible strategy might be the following. 

Proof strategy (PS):
Show that for any system of $n$ pinned balls, a sequence $\Gamma$ of pairs of balls (see Section \ref{notdef} for the definition and meaning) and arbitrarily small $\eps >0$, one can find a family of $n$ elastically colliding moving balls such that their total energy is equal to 1, the center of the $k$-th  moving ball stays within $\eps$ from the center of the $k$-th pinned ball over the time interval $[0,\eps]$, and the sequence of collisions of the moving balls over the time interval $[0,\eps]$ is $\Gamma$, i.e., the  pairs of moving balls collide in the same order as the corresponding pairs of pinned balls.

We will show in the following example that this strategy does not work. This does not rule out another link between upper bounds for the numbers of collisions of pinned and moving balls.

\begin{example}
Consider discs in the plane: $A=\calB((0,-2),1)$, $B=\calB((0,0),1)$, $C=\calB((-1,\sqrt{3}),1)$, and $D=\calB((1,\sqrt{3}),1)$. 
The pseudo-velocity of $A$ after the $j$-th collision will be denoted $v_A(j)$, $j\geq 0$, with $v_A(0)$ representing the initial velocity. The pseudo-velocities of other balls will be denoted in an analogous way. 
We start with the following pseudo-velocities: $v_A(0)=(0,2)$, $v_B(0) = (0,1)$, $v_C(0)=v_D(0)=(0,0)$. See Fig.~\ref{fig1}.

\begin{figure} \includegraphics[width=0.4\linewidth]{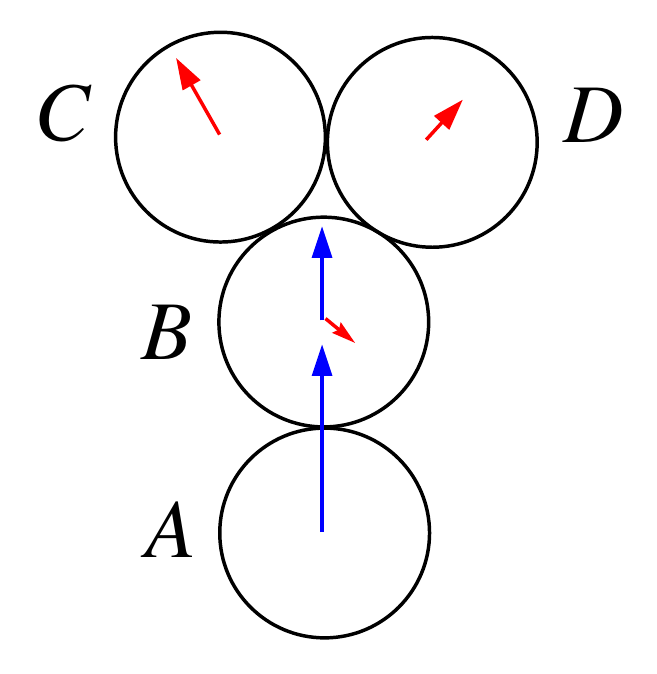}
\caption{
Blue arrows represent initial velocities of balls $A $ and $B$. Red arrows represent velocities of balls $B, C$ and $D$ after collisions of ball $B$ with balls $C$ and $D$, in this order. 
}
\label{fig1}
\end{figure}

We will consider the following sequence of collisions $\Gamma_1$: $BC$, $BD$, $AB$, $BC$. We will list velocities after consecutive collisions. We will give formulas only for the balls that collide. All other velocities remain unchanged if they are not listed on a given line. See Fig.~\ref{fig1}.
\begin{align*}
v_B(1) &= (\sqrt{3}/4, 1/4), \qquad v_C(1) = (-\sqrt{3}/4, 3/4),\\
v_B(2)&=(\sqrt{3}/8, -1/8)  ,\qquad v_D(2)=(\sqrt{3}/8, 3/8),\\
v_A(3)&=(0, -1/8)  ,\qquad  v_B(3)=(\sqrt{3}/8, 2),\\
v_B(4) &= (11 \sqrt{3}/32,43/32), \qquad v_C(4) = (-15 \sqrt{3}/32,45/32).
\end{align*}

Next consider the following sequence of collisions $\Gamma_2$: $BC$, $BD$, $AB$, $BD$, $BC$. The velocities after consecutive collisions are
\begin{align*}
\wt v_B(1) &= (\sqrt{3}/4, 1/4), \qquad \wt v_C(1) = (-\sqrt{3}/4, 3/4),\\
\wt v_B(2)&=(\sqrt{3}/8, -1/8)  ,\qquad \wt v_D(2)=(\sqrt{3}/8, 3/8),\\
\wt v_A(3)&=(0, -1/8)  ,\qquad  \wt v_B(3)=(\sqrt{3}/8, 6),\\
\wt v_B(4) &= (-9 \sqrt{3}/32,25/32), \qquad \wt v_D(4) = (17 \sqrt{3}/32,51/32),\\
\wt v_B(5) &= (-17 \sqrt{3}/64,47/64), \qquad \wt v_C(5) = (-17 \sqrt{3}/64,51/64).
\end{align*}

We will argue that $\Gamma_1$ cannot be approximated by a system of moving billiard balls in the sense of proof strategy (PS) outlined at the beginning of this section. 
Suppose that the balls can move and $\eps>0$ is very small.
Let $\be_1 = (-1/2, \sqrt{3}/2)$ and $\be_2 = (1/2, \sqrt{3}/2)$. These are unit vectors pointing from the center of $B$ in the directions of the centers of $C$ and $D$, resp.
Immediately after the second collision, 
balls $B$ and $D$ will touch so the distance between balls $B$ and $C$ will be greater than that between $B$ and $D$. 
We have 
\begin{align*}
\langle \be_1, v_C(3)-v_B(3)\rangle=-7 \sqrt{3}/16 >-13\sqrt{3}/16 =\langle \be_2, v_D(3)-v_B(3)\rangle
\end{align*}
so if  $B$ undergoes a collision with either $C$ or $D$ at time 4,  the collision $BD$ must precede $BC$. In other words, $\Gamma_2$ might be represented by moving balls as in (PS) but $\Gamma_1$ cannot.
\end{example}

\section{Comparison with earlier results}\label{comp}

We will compare the bound given in Theorem \ref{au21.7} with some bounds proved in \cite{fold}.

\begin{definition}\label{s21.1}
(i)
Consider a family of pinned balls $\bF$ and recall the definition of $z_{jk}$ from  \eqref{au20.2}.

Consider a graph $G_1=(\calV_1, \calE_1)$ associated with $\bF$, suppose that $(i_1, i_2) \in \calE_1$ and let $\calE_2 = \calE_1 \setminus \{(i_1, i_2)\}$.
Let $\alpha_*(G_1, (i_1, i_2))$ be the distance from $  z _{i_1 i_2}$ to the linear subspace spanned by $\left\{ z _{kj}, (k,j) \in \calE_2\}\right\}$.

We define $\alpha(\bF)$ 
to be the minimum of all \emph{strictly positive} values of 
$\alpha_*(G_1, (i_1, i_2))$ over all graphs $G_1=(\calV_1, \calE_1)$ associated with $\bF$ and all $(i_1, i_2)\in \calE_1$.

(ii) Let $\calX$ be the set of all points in the plane of the form
$( 2j, 2 k \sqrt{3})$ or $( 2j+1, (2 k+1) \sqrt{3})$ for integers $j$ and $k$. That is, $\calX$ is the set of vertices of a triangular lattice, assuming that we add edges between pairs of points at distance 2.

\end{definition}

The following theorem and corollary were proved in \cite{fold}.

\begin{theorem}\label{au21.6}
(\cite[Thm. 9.3]{fold})
The  number of collisions of $n$ pinned balls in $\R^d$, i.e., the  number of distinct vectors in the sequence $(v(t), t\geq 0)$, is bounded above by
\begin{align}\label{s28.2}
\left( \frac{ 2^{21/2} d n^5 }{\alpha(\bF)}\right)^{\tau_d n/2-1}.
\end{align}
\end{theorem}

\begin{corollary}\label{m5.1}
(\cite[Cor. 9.4]{fold})
(i) If the full graph associated with the family of $n$ pinned balls is a tree then the number of collisions is not greater than
\begin{align}\label{s28.3}
\left( 2^{10} d n^6 \right)^{\tau_d n/2-1}.
\end{align}

(ii) 
Suppose that $d=2$ and the centers of pinned balls $\bF$ belong to the triangular lattice $\calX$. Then the number of collisions is not greater than
\begin{align}\label{s28.4}
 \left( 10^6 n^{11/2}  4^{4n}\right)^{3 n-1}.
\end{align}

\end{corollary}

\begin{remark}\label{s11.2}

(i) It was indicated by an example in \cite{fold} that although $\alpha(\bF)$ is strictly positive by definition, $\alpha(\bF)$ can be arbitrarily small. Hence, \eqref{f16.1} is a better bound then \eqref{s28.2} for some configurations of pinned balls. Effective computation of  $\alpha(\bF)$ is very hard in specific cases, as attested by the arguments in \cite{fold}. Thus, the bound \eqref{f16.1} has an advantage of being totally explicit in all cases.

(ii) If the full full graph associated with the family of $n$ pinned balls is a tree then the number of half-spaces $H_{ij}$ defined in Section \ref{fold} is equal to $n-1$. Hence, we can replace the the expression $\tau_d n$ appearing in both \eqref{f16.1} and \eqref{s28.3} by $2n$.
We can also replace $d$ with $2n-1$ because the pseudo-velocities span  a $(2n-1)$-dimensional subspace.
The new bounds are 
\begin{align*}
 2^{-27}  2^{16n}
n^{-27/2}  n^{15n}
\qquad \text{
(improved \eqref{f16.1})},\\
\left( 2^{10} (2n-1) n^6 \right)^{n-2}
\leq 2^{11n-22}n^{-14} n^{7n}
 \qquad \text{
(improved \eqref{s28.3})}.
\end{align*}
Hence, for tree-like pinned ball configurations  the 
bound given in \cite{fold} is better than \eqref{f16.1} for large $n$.

(iii) If $d=2$ and the centers of pinned balls  belong to the triangular lattice $\calX$, and $n$ is large then the bound in \eqref{f16.1} is better than that in \eqref{s28.4} because 
$\tau_2=6$ and, therefore, the most significant factors in the two formulas are $n^{45n/2}$ and $4^{12n^2}$, resp.

(iv) There are both intuitive and theoretical
reasons to think that pinned balls model is closely related to the stage in the moving balls evolution that generates the largest number of collisions. For instance, for appropriate $c,\eps>0$, Theorem 1.3 in \cite{BuD18b} asserts that if a total of more than $n^{cn}$ collisions occur, then at least $n^{\eps n}$ of them will have to occur in an interval of time during which a subset of the balls form a very tight configuration. 

One can use the estimates in \eqref{s26.1}-\eqref{j19.2} as a gauge, to get a feeling of whether the estimates for  pinned balls might be close to optimal.
The most significant factor in \eqref{s26.1} is $n^{(3/2) n^2}$. This is smaller than $n^{(15/2) n^2 }$, the most significant factor in  \eqref{f16.1}, for large $n$, if we take $\ell = n^2/2$. In view of \eqref{s11.1},  if we use $\ell=\tau_d n/2$, the most significant factor in \eqref{f16.1} is $n^{(15/2)\tau_d n/2 }\leq n^{(15/4)2^{0.402d}  n}$ which is
 smaller than $n^{((3/2)5^d+9/2) n }$, the most significant factor in \eqref{j19.2},
for  large $n$ and $d$.

(v) A lower bound for the number of collisions of $n$ pinned balls is $n^3/27$ for $n\geq 3$ and $d\geq 2$. This is the same bound as the one for a family of moving balls, presented in \cite{BD}. The description of the main example and the arguments given in \cite{BD} for moving balls clearly show that the same bound applies to pinned balls. 

\end{remark}

\bibliographystyle{alpha}

\begin{thebibliography}{BFK98b}

\bibitem[ABD21]{fold}
Jayadev~S. Athreya, Krzysztof Burdzy, and Mauricio Duarte.
\newblock On pinned billiard balls and foldings.
\newblock {\em Indiana U. Math. J.}, 2021.
\newblock To appear, Arxiv:1807.08320.

\bibitem[BD19]{BD}
Krzysztof Burdzy and Mauricio Duarte.
\newblock A lower bound for the number of elastic collisions.
\newblock {\em Comm. Math. Phys.}, 372(2):679--711, 2019.

\bibitem[BD20]{BuD18b}
Krzysztof Burdzy and Mauricio Duarte.
\newblock On the number of hard ball collisions.
\newblock {\em J. Lond. Math. Soc. (2)}, 101(1):373--392, 2020.

\bibitem[Bez10]{Bez}
K{\'a}roly Bezdek.
\newblock {\em Classical topics in discrete geometry}.
\newblock CMS Books in Mathematics/Ouvrages de Math\'ematiques de la SMC.
  Springer, New York, 2010.

\bibitem[BFK98a]{BFK5}
D.~Burago, S.~Ferleger, and A.~Kononenko.
\newblock A geometric approach to semi-dispersing billiards.
\newblock {\em Ergodic Theory Dynam. Systems}, 18(2):303--319, 1998.

\bibitem[BFK98b]{BFK2}
D.~Burago, S.~Ferleger, and A.~Kononenko.
\newblock Unfoldings and global bounds on the number of collisions for
  generalized semi-dispersing billiards.
\newblock {\em Asian J. Math.}, 2(1):141--152, 1998.

\bibitem[BFK98c]{BFK1}
D.~Burago, S.~Ferleger, and A.~Kononenko.
\newblock Uniform estimates on the number of collisions in semi-dispersing
  billiards.
\newblock {\em Ann. of Math. (2)}, 147(3):695--708, 1998.

\bibitem[BFK00]{BFK3}
D.~Burago, S.~Ferleger, and A.~Kononenko.
\newblock A geometric approach to semi-dispersing billiards.
\newblock In {\em Hard ball systems and the {L}orentz gas}, volume 101 of {\em
  Encyclopaedia Math. Sci.}, pages 9--27. Springer, Berlin, 2000.

\bibitem[BFK02]{BFK4}
D.~Burago, S.~Ferleger, and A.~Kononenko.
\newblock Collisions in semi-dispersing billiard on {R}iemannian manifold.
\newblock In {\em Proceedings of the {I}nternational {C}onference on {T}opology
  and its {A}pplications ({Y}okohama, 1999)}, volume 122, pages 87--103, 2002.

\bibitem[BI18]{BuI18}
Dmitri {Burago} and Sergei {Ivanov}.
\newblock {Examples of exponentially many collisions in a hard ball system}.
\newblock September 2018.
\newblock arXiv:1809.02800.

\bibitem[Bur22]{uphard}
Krzysztof Burdzy.
\newblock An improved upper bound on the number of billiard ball collisions.
\newblock {\em Commun. Math. Phys.}, 391:107–117, 2022.

\bibitem[CI04]{IllnerChen}
Xinfu Chen and Reinhard Illner.
\newblock Finite-range repulsive systems of finitely many particles.
\newblock {\em Arch. Ration. Mech. Anal.}, 173(1):1--24, 2004.

\bibitem[Ill89]{Illner89}
Reinhard Illner.
\newblock On the number of collisions in a hard sphere particle system in all
  space.
\newblock {\em Transport Theory Statist. Phys.}, 18(1):71--86, 1989.

\bibitem[Ill90]{Illner90}
Reinhard Illner.
\newblock Finiteness of the number of collisions in a hard sphere particle
  system in all space. {II}. {A}rbitrary diameters and masses.
\newblock {\em Transport Theory Statist. Phys.}, 19(6):573--579, 1990.

\bibitem[Vas79]{Vaser79}
L.~N. Vaserstein.
\newblock On systems of particles with finite-range and/or repulsive
  interactions.
\newblock {\em Comm. Math. Phys.}, 69(1):31--56, 1979.

\end{thebibliography}
\def\cprime{$'$}

\end{document}